\documentclass[10pt]{article}
\usepackage{amsmath,amssymb,amsthm}

\newtheorem{theorem}{Theorem}[section]
\newtheorem{proposition}[theorem]{Proposition}
\newtheorem{lemma}[theorem]{Lemma}

\theoremstyle{definition}
\newtheorem{definition}[theorem]{Definition}

\theoremstyle{remark}
\newtheorem{remark}[theorem]{Remark}

\newcommand{\N}{\mathbb{N}}
\newcommand{\Z}{\mathbb{Z}}

\newcommand{\C}{\mathbb{C}}

\newcommand{\cH}{\cal H}

\newcommand{\Zjs}{{\cal Z}}

\DeclareMathOperator{\id}{id}

\DeclareMathOperator{\Ad}{Ad}

\begin{document}
\title{Strict comparison and $\Zjs$-absorption of nuclear C$^*$-algebras}
\author{Hiroki Matui \\
Graduate School of Science \\
Chiba University \\
Inage-ku, Chiba 263-8522, Japan 
\and 
Yasuhiko Sato \\
Graduate School of Science \\
Kyoto University \\
Sakyo-ku, Kyoto 606-8502, Japan}
\date{}

\maketitle

\begin{abstract}   
For any unital separable simple infinite-dimensional nuclear C$^*$-algebra with finitely many extremal traces, we prove that $\Zjs$-absorption, strict comparison, and property (SI) are equivalent. We also show that 
any unital separable simple nuclear C$^*$-algebra with tracial rank zero 
is approximately divisible, and hence is $\mathcal{Z}$-absorbing. 

\end{abstract}

\section{Introduction}\label{Sec1}

X. Jiang and H. Su \cite{JS} constructed 
a unital separable simple infinite-dimensional nuclear C$^*$-algebra $\mathcal{Z}$, 
called the Jiang-Su algebra, 
whose $K$-theoretic invariant is isomorphic to 
that of the complex numbers. 
The Jiang-Su algebra has recently become to play a central role 
in Elliott's classification program for nuclear C$^*$-algebras. 
We say that a unital C$^*$-algebra is $\mathcal{Z}$-absorbing 
if $A\cong A\otimes\mathcal{Z}$. 
 H. Lin, Z. Niu and W. Winter proved that 
certain $\mathcal{Z}$-absorbing C$^*$-algebras are classified 
by their ordered $K$-groups \cite{LN, WinClassification}. 
Indeed, all classes of unital simple nuclear C$^*$-algebras 
for which Elliott's classification conjecture have been confirmed 
consist of $\mathcal{Z}$-absorbing algebras. 
One may view $\mathcal{Z}$ as being the stably finite analogue 
of the Cuntz algebra $\mathcal{O}_\infty$. 
W. Winter also showed that $\mathcal{Z}$ is the initial object 
in the category of strongly self-absorbing C$^*$-algebras 
\cite{WinStrSelfAbsorbing}. 

In view of this, 
it is desirable to characterize $\mathcal{Z}$-absorbing C$^*$-algebras 
in various manners. 
In 2008, A. S. Toms and W. Winter conjectured that 
the properties of strict comparison, finite nuclear dimension, 
and $\mathcal{Z}$-absorption are equivalent 
for unital separable simple infinite-dimensional nuclear C$^*$-algebras 
(see \cite{Toms,WZ} for example). 
M. R\o rdam proved that 
$\mathcal{Z}$-absorption implies strict comparison 
for unital simple exact C$^*$-algebras \cite{Ror}. 
W. Winter showed that 
any unital separable simple infinite-dimensional C$^*$-algebra 
with finite nuclear dimension 
is $\mathcal{Z}$-absorbing \cite{WinNucDimension}. 
In the present paper 
we provide another partial answer to the conjecture above. 
Namely, it will be shown that 
strict comparison implies $\mathcal{Z}$-absorption 
under the assumption that the algebra has finitely many extremal traces.

The following is the main result of this paper.
\begin{theorem}\label{ThmMain}
Let $A$ be a unital separable simple infinite-dimensional nuclear C$^*$-algebra with finitely many extremal traces. Then the following are equivalent:
\begin{enumerate}
\item $A\otimes\Zjs\cong A$.
\item $A$ has strict comparison.
\item Any completely positive map from $A$ to $A$ can be excised in small central sequences.
\item $A$ has property (SI).
\end{enumerate}
\end{theorem}
Here, we recall the definition of strict comparison. 
In this paper we denote by $A_+$ the positive cone of $A$ and 
by $T(A)$ the set of tracial states on $A$. 
We define the dimension function $d_{\tau}$ associated to $\tau\in T(A)$ 
by $d_{\tau}(a)=\lim_{n\rightarrow\infty}\tau(a^{1/n})$ for $a\in M_k(A)_+$, 
where $\tau$ is regarded as an unnormalized trace on $M_k(A)$. 
We say that a separable nuclear C$^*$-algebra $A$ has {\it strict comparison} 
if for $a,b\in M_k(A)_+$ with $d_{\tau}(a)<d_{\tau}(b)$ for any $\tau\in T(A)$ 
there exist $r_n\in M_k(A)$, $n\in\N$ such that $r_n^*br_n\rightarrow a$. 
The definition of excision in small central sequences 
is given in Definition \ref{DefExcision}, 
and the definition of property (SI) is given in Definition \ref{DefSI}. 
As mentioned above, 
(i)$\Rightarrow$(ii) was proved by M. R\o rdam \cite[Corollary 4.6]{Ror} 
without assuming that $A$ has finitely many extremal tracial states. 
The implication (iii)$\Rightarrow$(iv) is immediate from the definitions 
and do not need the assumption of finitely many extremal traces. 
We use the full assumption on $A$ 
for the implications (ii)$\Rightarrow$(iii) and (iv)$\Rightarrow$(i). 
It will be also shown that 
(i) implies (iii) and (iv) 
without the assumption of finitely many extremal tracial states 
in Theorem \ref{ThmSI}. 
In Section 5, using the same method, 
we shall show approximate divisibility of 
unital separable simple nuclear C$^*$-algebras with tracial rank zero.

The main technical device in this paper is 
excision of completely positive maps. 
In \cite{AAP}, C. A. Akemann, J. Anderson and G. K. Pedersen proved that 
any pure state on a C$^*$-algebra can be excised 
by positive norm one elements. 
By using their result, 
E. Kirchberg obtained a Stinespring dilation type theorem 
for unital nuclear completely positive maps 
from a unital purely infinite simple C$^*$-algebra to itself. 
This theorem is one of technical cornerstones 
in the proof of Kirchberg's celebrated embedding theorem 
for exact C$^*$-algebras \cite{Kir,KP}. 
In this article, by using the result of \cite{AAP}, 
we will establish a similar `dilation' type result 
for completely positive maps in the setting of stably finite C$^*$-algebras. 
To this end, we have to work with central sequences 
and to take into account the values of traces on them 
(Definition \ref{DefExcision}). 

The other ingredient in this paper is property (SI). 
The idea of property (SI) originates with A. Kishimoto 
(see \cite[Lemma 3.6]{K98JOT}). 
Using it, he proved that 
certain automorphisms of AT algebras have the Rohlin property. 
(See \cite{M10CMP,SatIJM} for further developments.) 
In \cite{MS,Sat}, property (SI) was used to show 
$\mathcal{Z}$-absorption of crossed products by strongly outer actions. 
The main theorem in the present paper implies that 
this property is not so restrictive but is shared 
by `many' stably finite nuclear C$^*$-algebras. 

We recall the notion of central sequence algebras of C$^*$-algebras. 
Let $A$ be a separable C$^*$-algebra. 
Set 
\[
A^\infty=\ell^\infty(\N,A)/\{(a_n)_n\in\ell^\infty(\N,A)\mid
\lim_{n\to\infty}\lVert a_n\rVert=0\}. 
\]
We identify $A$ with the C$^*$-subalgebra of $A^\infty$ 
consisting of equivalence classes of constant sequences. 
We let 
\[
A_\infty=A^\infty\cap A'
\]
and call it the central sequence algebra of $A$. 
A sequence $(x_n)_n\in\ell^\infty(\N,A)$ is called a central sequence 
if $\lVert[a,x_n]\rVert\to0$ as $n\to\infty$ for all $a\in A$. 
A central sequence is a representative of an element in $A_\infty$.

\section{Excision in small central sequences}\label{Sec2}
In this section, we prove Proposition \ref{PropMain}, which plays an important role in Section \ref{Sec3}. 
 
\begin{definition}\label{DefExcision}
Let $A$ be a separable C$^*$-algebra with $T(A)\neq \emptyset$, and let $\varphi:A\rightarrow A$ be a completely positive map. 
We say that {\it $\varphi$ can be excised in small central sequences} when  for any central sequences 
$(e_n)_n$ and $(f_n)_n$ of positive contractions in $A$ satisfying 
\[ \lim_{n\rightarrow\infty}\max_{\tau\in T(A)} \tau(e_n)=0, \quad \lim_{m\rightarrow\infty}\liminf_{n\rightarrow\infty} \min_{\tau\in T(A)} \tau (f_n^m) > 0, \]
there exist $s_n\in A$, $n\in\N$ such that 
\[\lim_{n\rightarrow\infty} \| s_n^*as_n- \varphi(a) e_n\| =0,{\rm \ for\ any\ } a\in A, \quad \lim_{n\rightarrow\infty} \|f_ns_n-s_n\| =0.\]    

\end{definition}

The following proposition is our main tool for the proof of (ii)$\Rightarrow$(iii) of Theorem \ref{ThmMain}. This may be thought of as a stably finite analogue  of Kirchberg's Stinespring type theorem \cite{Kir} (see also \cite[Proposition 1.4]{KP}).

\begin{proposition}\label{PropMain}
Let $A$ be a unital separable simple infinite-dimensional C$^*$-algebra 
with $T(A)\neq\emptyset$. 
Suppose that $A$ has strict comparison. 
Let $\omega$ be a state on $A$ and let $c_i$, $d_i\in A$, $i=1,2,\dots,N$. 
Then the completely positive map $\varphi:A\rightarrow A$ defined by 
\[
\varphi(a)=\sum_{i,j=1}^N\omega(d_i^*ad_j)c_i^*c_j, \quad a\in A,
\]
can be excised in small central sequences.
\end{proposition}
  
In order to prove this proposition, 
we need a couple of lemmas.
\begin{lemma}\label{Lemfn}
Let $A$ be a separable  C$^*$-algebra with $T(A)\neq \emptyset$. 
For any central sequence $(f_n)_n$ of positive contractions in $A$, 
there exists a central sequence 
$(\tilde f_n)_n$ of positive contractions in $A$ such that 
$(\tilde f_nf_n)_n= (\tilde f_n)_n$ and 
\[ \lim_{m\rightarrow\infty} \liminf_{n\rightarrow\infty} \min_{\tau\in T(A)} \tau({\tilde f_n}^m) = \lim_{m\rightarrow\infty} \liminf_{n\rightarrow\infty} \min_{\tau\in T(A)} \tau(f_n^m).\]
\end{lemma}
\begin{proof}
We can find a natural number $N_m\in\N$ such that 
\[
\liminf_{n\to\infty}\min_{\tau\in T(A)}\tau(f_n^m)
<\min_{\tau\in T(A)}\tau(f_l^m)+\frac{1}{m}
\]
holds for every $l>N_m$. 
We may assume $N_m<N_{m+1}$. 
Define a sequence $(m_n)_n$ of natural numbers 
so that $m_n=m$ when $N_m<n\leq N_{m+1}$. 
Note that $(m_n)_n$ is an increasing sequence such that $m_n\to\infty$ and 
\[
\liminf_{n\to\infty}\min_{\tau\in T(A)}\tau(f_n^{m_n})
\geq\lim_{m\to\infty}\liminf_{n\to\infty}\min_{\tau\in T(A)}\tau(f_n^m). 
\]
Let $(\tilde m_n)_n$ be a sequence of natural numbers such that 
$\tilde m_n\to\infty$, $\tilde m_n\leq m_n^{1/2}$ and 
$(f^{\tilde m_n}_n)_n$ is a central sequence. 
Let $\tilde f_n=f_n^{\tilde m_n}$. 
It is easy to see $(f_n\tilde f_n)_n=(\tilde f_n)_n$.
Also, we have 
\begin{align*}
\lim_{m\to\infty}\liminf_{n\to\infty}\min_{\tau\in T(A)}\tau(f_n^m)
&\geq\lim_{m\to\infty}\liminf_{n\to\infty}
\min_{\tau\in T(A)}\tau(\tilde f_n^m)\\
&=\lim_{m\to\infty}\liminf_{n\to\infty}
\min_{\tau\in T(A)}\tau(f_n^{m\tilde m_n})\\
&\geq\liminf_{n\to\infty}\min_{\tau\in T(A)}\tau(f_n^{m_n})
\geq\lim_{m\to\infty}\liminf_{n\to\infty}\min_{\tau\in T(A)}\tau(f_n^m). 
\end{align*}
\end{proof}

\begin{lemma}\label{LemAlpha}
Let $A$ be a unital separable simple C$^*$-algebra with $T(A)\neq \emptyset$ and let $a\in A$ be a non-zero positive element. Then there exists $\alpha >0$ such that 
\[\alpha \liminf_{n\rightarrow\infty} \min_{\tau\in T(A)} \tau(f_n) \leq \liminf_{n\rightarrow\infty} \min_{\tau\in T(A)} \tau(f_n^{1/2} af_n^{1/2}),\]
for any central sequence $(f_n)_n$ of positive contractions in $A$.
\end{lemma}
\begin{proof}
Since $A$ is unital and simple, there exist $v_1, v_2,\dots, v_m \in A$ such that
$\displaystyle \sum_{i=1}^m v_i^*av_i =1$. Set $\alpha= (\sum_i \|v_i\|^2)^{-1} > 0$. Then we have
\begin{align*}
\liminf_{n\rightarrow\infty} \min_{\tau\in T(A)} \tau(f_n)
&=\liminf_n\min_{\tau}\sum_{i=1}^m\tau(v_i^*av_i f_n) \\
&= \liminf_n\min_{\tau} \sum_{i=1}^m \tau(v_i^*{a^{1/2}} f_n {a^{1/2}} v_i)\\
&= \liminf_n \min_{\tau} \sum_{i=1}^m \tau (f_n^{1/2} {a^{1/2}} v_i{v_i}^* {a^{1/2}} f_n^{1/2} )\nonumber \\
&\leq \alpha^{-1} \liminf_n \min_{\tau} \tau (f_n^{1/2} af_n^{1/2}).\nonumber
\end{align*}
\end{proof}

\begin{lemma}\label{LemBddrn}
Let $A$ be a unital separable simple C$^*$-algebra with $T(A)\neq \emptyset$. 
Suppose that $A$ has strict comparison. 
Let $(e_n)_n$ and $(f_n)_n$ be as in Definition \ref{DefExcision}. 
Then for any norm one positive element $a\in A$, 
there exists a sequence $(r_n)_n$ in $A$ such that 
\[
\lim_{n\to\infty}\lVert r_n^*f_n^{1/2}a f_n^{1/2}r_n-e_n\rVert=0,\quad 
\limsup_{n\to\infty}\lVert r_n\rVert= \limsup_{n\rightarrow\infty}\lVert e_n\rVert^{1/2}. 
\]
\end{lemma}
\begin{proof}
By Lemma 3.2 (i) in \cite{Sat}, 
we may assume $\displaystyle\lim_n\max_\tau d_\tau(e_n) =0$. 
Set 
\[
c=\lim_{m\to\infty}\liminf_{n\to\infty}\min_{\tau\in T(A)}\tau(f_n^m)>0. 
\]
Take $\varepsilon>0$. 
It suffices to show that there exist $r_n\in A$, $n\in\N$ such that 
\[
\limsup_{n\to\infty}\lVert r_n^*f_n^{1/2}af_n^{1/2}r_n-e_n\rVert
\leq\varepsilon,\quad 
\lim_{n\to\infty}\lVert r_n^*r_n - e_n \rVert =0. 
\]

As $\lVert a\rVert=1$, using continuous functional calculus, 
we get non-zero positive contractions $a_0,a_1\in A$ such that 
$\lVert a_0-a\rVert\leq\varepsilon$ and 
$a_1\leq a_0^m$ for all $m\in\N$. 
Applying Lemma \ref{LemAlpha} to $a_1\in A_+\setminus\{0\}$, we obtain $\alpha>0$. 
Then for any $m\in\N$ it follows that 
\begin{align*}
\alpha c&\leq\alpha\liminf_{n\to\infty}\min_{\tau\in T(A)}\tau(f_n^m)\\
&\leq\liminf_{n\to\infty}\min_{\tau\in T(A)}\tau(f_n^{m/2}a_1f_n^{m/2})\\
&\leq\liminf_{n\to\infty}\min_{\tau\in T(A)}\tau(f_n^{m/2}a_0^mf_n^{m/2}). 
\end{align*}
Put $b_n=f_n^{1/2}a_0f_n^{1/2}$. 
Since $(f_n)_n$ is central, one has 
\[
\liminf_{n\to\infty}\min_{\tau\in T(A)}\tau(b_n^m)
=\liminf_{n\to\infty}\min_{\tau\in T(A)}\tau(f_n^{m/2}a_0^mf_n^{m/2})
\geq\alpha c
\]
for any $m\in\N$. Then we have an increasing sequence $m_n \in \N$ of natural numbers such that $m_n \rightarrow \infty$ and $\displaystyle \liminf_n \min_{\tau} \tau (b_n^{m_n}) \geq \alpha c$.

For $\delta >0$, define a continuous function $g_{\delta}\in C([0,1])$ by $g_{\delta}(t) = \max \{ 0, \delta^{-1}(t-1+\delta)\}$. Let $\varepsilon_n>0$, $n\in\N$ be a decreasing sequence such that $\varepsilon_n \rightarrow 0$ and $(1-\varepsilon_n)^{m_n}\to 0$. Then we have 
\begin{align*}
\liminf_{n\rightarrow\infty} \min_{\tau\in T(A)} d_{\tau}(g_{\varepsilon_n} (b_n))
&\geq \liminf_n\min_{\tau} \tau(g_{\varepsilon_n}(b_n))  \\
&\geq \liminf_n\min_{\tau} \tau(b_n^{m_n}) - (1-\varepsilon_n)^{m_n} \\
&= \liminf_n\min_{\tau} \tau (b_n^{m_n}) \geq \alpha c >0.
\end{align*}

Because $A$ has strict comparison, 
we can find a sequence $(q_n)_n$ in $A$ such that 
\[
\lim_{n\to\infty}\lVert q_n^*g_{\varepsilon_n}(b_n)q_n-e_n\rVert=0. 
\]
Note that $(q_n)_n$ is not necessarily bounded. 
We define $r_n = g_{\varepsilon_n}^{1/2}(b_n)q_n$ for $n\in\N$.
Then it follows that
\begin{align*}
\lVert(1-b_n)r_n\rVert
&\leq \varepsilon_n\lVert r_n\rVert\to 0, \\
\| r_n^* b_n r_n -e_n \| 
&\leq \lVert r_n^* (b_n-1) r_n\rVert + \lVert r_n^*r_n -e_n\rVert\to \ 0,
\end{align*}
 as $n\rightarrow\infty$.
Consequently we have 
\[
\limsup_{n\to\infty}\lVert r_n^*f_n^{1/2}af_n^{1/2}r_n-e_n\rVert
\leq\limsup_{n\to\infty}
\lVert r_n^*f_n^{1/2}a_0f_n^{1/2}r_n-e_n\rVert+\varepsilon
=\varepsilon. 
\]
\end{proof}

Now we are ready to prove Proposition \ref{PropMain}. 
\begin{proof}[Proof of Proposition \ref{PropMain}]
Let $\varphi:A\to A$ be as in the statement. 
Replacing $c_i$ and $d_i$ with $c_i/\|c_i\|$ and $\|c_i\|d_i$ we may assume 
$\|c_i\|\leq 1$. 
Let $F$ be a finite subset of the unit ball of $A$ and let $\varepsilon>0$. 
It suffices to show that there exist $s_n\in A$, $n\in\N$ such that 
\[
\limsup_{n\to\infty}\lVert s_n^*xs_n-\varphi(x)e_n\rVert<\varepsilon
\quad\text{and}\quad 
\lim_{n\to\infty}\lVert f_ns_n-s_n\rVert=0
\]
for any $x\in F$. 
Set $G=\{d_i^*xd_j\in A\mid x\in F,\ i=1,2,\dots,N\}$ and 
$\delta=\varepsilon/N^2$. 

Since $A$ is unital simple infinite-dimensional, by Glimm's lemma, 
any state on $A$ can be approximated by pure states in the weak$*$-topology. 
Hence we may assume that $\omega$ is a pure state on $A$. 
By \cite[Proposition 2.2]{AAP}, there exists $a\in A_+$ such that 
$\lVert a\rVert=1$ and 
$\lVert a(\omega(x) -x)a\rVert<\delta$ for every $x\in G$. 
Let $(e_n)_n$ and $(f_n)_n$ be as in Definition \ref{DefExcision}. 
By Lemma \ref{Lemfn}, 
we obtain a central sequence $(\tilde f_n)_n$ of positive contractions in $A$ 
satisfying $(\tilde f_nf_n)_n=(\tilde f_n)_n$ and 
\[
\lim_{m\to\infty}\liminf_{n\to\infty}\min_{\tau\in T(A)}\tau(\tilde f_n^m)
=\lim_{m\to\infty}\liminf_{n\to\infty}\min_{\tau\in T(A)}\tau(f_n^m)>0. 
\]
Applying Lemma \ref{LemBddrn} to $(e_n)_n$, $(\tilde f_n)_n$ and $a^2$, 
we obtain $r_n\in A$, $n\in\N$ satisfying 
\[
\lim_{n\to\infty}
\lVert r_n^*\tilde f_n^{1/2}a^2\tilde f_n^{1/2}r_n-e_n\rVert=0,\quad 
\limsup_{n\to\infty}\lVert r_n\rVert\leq 1. 
\]

We define  
\[ s_n = \sum_{i=1}^N d_ia {\tilde f_n}^{1/2} r_n c_i,\quad n\in\N.\]
Since $(f_n)_n$ is central and $(r_n)_n$ is bounded it follows that 
\begin{align*}
\limsup_n\| f_ns_n- s_n\| 
&\leq \limsup_n \sum_{i=1}^N \| (1-f_n)d_ia{\tilde f_n}^{1/2}\|\cdot\|r_n\| \\ 
&=\limsup_n \sum_{i=1}^N \|d_i a(1-f_n){\tilde f_n}^{1/2}\| =0.
\end{align*}
Also, for any $x\in F$ we have 
\begin{align*}
 \limsup_{n\rightarrow\infty} \|s_n^* xs_n - \varphi(x) e_n \| 
&= \limsup_n \left\lVert \sum_{i,j=1}^N c_i^*(r_n^* {\tilde f_n}^{1/2}ad_i^* xd_j a {\tilde f_n}^{1/2} r_n - \omega(d_i^*xd_j) e_n)c_j \right\rVert \nonumber \\
&\leq \limsup_n \sum_{i,j=1}^N\left\lVert r_n^* {\tilde f_n}^{1/2} a d_i^* x d_j a{\tilde f_n}^{1/2} r_n - \omega(d_i^* xd_j)e_n\right\rVert \nonumber \\
&= \limsup_n  \sum_{i,j=1}^N \left\lVert r_n^* {\tilde f_n}^{1/2} (a d_i^* x d_j a - \omega(d_i^*xd_j)a^2){\tilde f_n}^{1/2} r_n \right\rVert \nonumber \\
&\leq  \limsup_n \sum_{i,j=1}^N \left\lVert a(d_i^* xd_j - \omega (d_i^*xd_j))a \right\rVert < \varepsilon. \nonumber
\end{align*}
\end{proof}
\begin{remark}
In the argument above, 
the assumption of strict comparison is used in the proof of Lemma \ref{LemBddrn}. 
But, it should be pointed out that 
we need much less than the full strength of strict comparison. 
Indeed, what we used in the proof of Lemma \ref{LemBddrn} is as follows: 
if $(e_n)_n$ and $(f_n)_n$ are sequences of positive contractions in $A$ 
satisfying 
\[
\lim_{n\to\infty}\max_{\tau\in T(A)}d_\tau(e_n)=0,\quad 
\liminf_{n\to\infty}\min_{\tau\in T(A)}d_\tau(f_n)>0, 
\]
then there exists a sequence $(r_n)_n$ in $A$ such that 
\[
\lim_{n\to\infty}\lVert r_n^*f_nr_n-e_n\rVert=0. 
\]
\end{remark}

\section{ \texorpdfstring{Proof of (ii) $\Rightarrow$ (iii) of Theorem \ref{ThmMain}}{Proof of (ii) $\Rightarrow$ (iii) of Theorem 1.1}}\label{Sec3}

In this section, we give a proof of (ii)$\Rightarrow$(iii) of Theorem \ref{ThmMain}, 
by using Proposition \ref{PropMain}. 
We begin with the following well-known fact.

\begin{lemma}\label{LemCPmap}
Let $A$ be a unital separable simple infinite-dimensional nuclear 
C$^*$-algebra, 
and let $\omega$ be a pure state of $A$. 
Then any completely positive map from $A$ to $A$ can be approximated 
in the pointwise norm topology 
by completely positive maps $\varphi$ of the form 
\[
\varphi(a)=\sum_{l=1}^N\sum_{i,j=1}^N
\omega(d_i^*ad_j)c_{l,i}^*c_{l,j},\quad a\in A,
\]
where $c_{l,i},d_i\in A$, $l,i=1,2,\dots,N$. 
\end{lemma}

\begin{proof} 
Let  $\rho : A\rightarrow M_N$ and $\sigma : M_N\rightarrow A$ be completely positive maps. 
Because $A$ is nuclear, any completely positive map is approximated by completely positive maps which factor through full matrix algebras. Thus it suffices to show that $\sigma\circ\rho$ can be approximated in the pointwise norm topology by completely positive maps $\varphi$ as in the lemma. Replacing $\rho$ and $\sigma$ by $\rho(1_A)^{-1/2}\rho(\cdot)\rho(1_A)^{-1/2}$ and $\sigma(\rho(1_A)^{1/2}\cdot \rho (1_A)^{1/2})$ with inverses taken in the respective hereditary subalgebra, we may assume that $\rho $ is unital.

We denote by $(\pi, \cH, \xi)$ the GNS representation associated with $\omega$. Since $A$ is  unital separable simple infinite-dimensional, $\pi(A)$ does not contain non-zero compact operators on $\cH$.  Applying Voiculescu's theorem (see  \cite[Theorem 1.7.8]{BO} for example)  to the unital completely positive map $\rho\circ\pi^{-1}: \pi(A) \rightarrow M_N$ we can find isometries $V_n :\C^N\rightarrow \cH$, $n\in\N$ such that 
\[ \lim_n \| \rho(a) - V_n^* \pi(a) V_n \| =0, \]
for any $a\in A$. Let $\{e_1, e_2,\dots,e_N\}$ be a basis for $\C^N$ and set $\xi_{i,n}= V_ne_i \in \cH$.  By Kadison's transitivity theorem
we obtain $d_{i,n}\in A$, $i=1,2,\dots,N$, $n\in\N$ such that $\pi(d_{i,n})\xi = \xi_{i,n}$. Then we have
\[ \omega(d_{i,n}^* ad_{j,n})=(\pi(a) \xi_{j,n} | \xi_{i,n})_{\cH} = (V_n^* \pi(a) V_n e_i | e_j),\]
for $i, j=1,2,\dots,N$, and $a\in A$, which implies 
\[ \lim_n \|\rho(a) - [ \omega(d_{i,n}^*ad_{j,n})]_{i,j} \| =0,\quad a\in A.\]

Let $e_{i,j}$ be the standard matrix units for $M_N$. 
Since $\sigma:M_N\rightarrow A$ is a completely positive map, 
the matrix $[\sigma(e_{i,j})]_{i,j}\in M_N(A)$ is positive 
(see \cite[Proposition 1.5.12]{BO} for example). 
Hence there exist $c_{l,j}\in A$, $l,j=1,2,\dots,N$ 
such that $\sigma(e_{i,j})=\sum_{l=1}^Nc_{l,i}^*c_{l,j}$. 
\end{proof}

The proof of the following lemma relies on A. Kishimoto's technique 
used in the proof of 2$\Rightarrow$1 of Theorem 4.5 in \cite{Kis}. 
For a state $\omega$ on $A$, 
we define the seminorm $\lVert\cdot\rVert_\omega$ 
by $\lVert a\rVert_\omega=\omega(a^*a)^{1/2}$ for $a\in A$.

\begin{lemma}\label{LemOrthogonality}
Let $\omega$ be a state on a unital separable C$^*$-algebra $A$ 
and let $k\in\N$. 
Let $(e_n)_n$ be a central sequence of positive contractions in $A$ and 
let $(u_n)_n$ be a central sequence of unitaries in $A$. 
If 
\[
\lim_{n\to\infty}\lVert\Ad u_n^i(e_n)e_n\rVert_\omega=0
\]
holds for every $i=1,2,\dots,k-1$, then 
there exists a central sequence $(e'_n)_n$ of positive contractions in $A$ 
such that 
\[
e'_n\leq e_n,\quad 
\lim_{n\to\infty}\omega(e_n-e'_n)=0,\quad\text{and}\quad 
\lim_{n\to\infty}\lVert\Ad u_n^i(e'_n)e'_n\rVert=0
\]
for every $i=1,2,\dots,k-1$. 
\end{lemma}

\begin{proof}
For $m\in\N$, we let $f_m$ denote the continuous function on $[0,\infty)$ 
defined by $f_m(t)=\min\{1,mt\}$. 
Define central sequences $(g_n)_n$ and $(e'_{m,n})_n$ by 
\[
g_n=e_n^{1/2}\left(\sum_{i=1}^{k-1}\Ad u_n^i(e_n)\right)e_n^{1/2}
\]
and 
\[
e'_{m,n}=e_n^{1/2}(1-f_m(g_n))e_n^{1/2}. 
\]
Note that $e'_{m,n}\leq e_n$ for any $m,n\in\N$. 
By the assumption of $e_n$ and $u_n$, for any $j\in\N$ it follows that 
\begin{align*}
\omega(e_n^{1/2}g_n^je_n^{1/2})
&\leq\lVert g_n\rVert^{j-1}\omega(e_n^{1/2}g_ne_n^{1/2})\\
&\leq(k-1)^{j-1}\sum_{i=1}^{k-1}\omega(e_n\Ad u_n^i(e_n)e_n)\\
&\leq(k-1)^{j-1}\sum_{i=1}^{k-1}\lVert\Ad u_n^i(e_n)e_n\rVert_{\omega}\to0,\quad 
n\to\infty. 
\end{align*} 
Then we have 
\[
\omega(e_n-e'_{m,n})
=\omega(e_n^{1/2}f_m(g_n)e_n^{1/2})\to0,\quad n\to\infty
\]
for any $m\in\N$. 
Furthermore, for $i=1,2,\dots,k-1$ we have 
\begin{align*}
\lVert\Ad u_n^i(e'_{m,n})e'_{m,n}\rVert^2
&\leq\lVert e'_{m,n}\Ad u_n^i(e'_{m,n})e'_{m,n}\rVert\\
&\leq\left\lVert e'_{m,n}
\sum_{i=1}^{k-1}\Ad u_n^i(e_n)e'_{m,n}\right\rVert\\
&=\left\lVert e_n^{1/2}(1-f_m(g_n))e_n^{1/2}
\sum_{i=1}^{k-1}\Ad u_n^i(e_n)e_n^{1/2}(1-f_m(g_n))e_n^{1/2}\right\rVert\\
&\leq\lVert(1-f_m(g_n))g_n\rVert<1/m. 
\end{align*}
Since $A$ is separable and $(e'_{m,n})_n$ is a central sequence, 
we can find an increasing sequence $(m_n)_n$ of natural numbers such that 
$m_n\to\infty$, $\omega(e_n-e'_{m_n,n})\to0$ and 
$(e'_{m_n,n})_n$ is a central sequence. 
Therefore $e'_n=e'_{m_n,n}$, $n\in\N$ satisfy the desired conditions. 
\end{proof}

In the proof of the following lemma, we use \cite[Lemma 2.1]{Sat2}. 
We remark that 
this lemma in \cite{Sat2} heavily depends on 
U. Haagerup's theorem \cite[Theorem 3.1]{Haa}, 
which says that any nuclear C$^*$-algebra has a virtual diagonal 
in the sense of B. E. Johnson \cite{Joh}. 

For the definition of order zero maps, 
the reader should see \cite[Section 1]{WZ}.

\begin{lemma}\label{LemOrderzeroCP}
Let $A$ be a unital separable simple infinite-dimensional nuclear C$^*$-algebra 
with finitely many extremal tracial states. 
For any $k\in\N$, 
there exist a completely positive contractive order zero map 
$\psi:M_k\to A_\infty$ and 
a central sequence $(c_n)_n$ of positive contractions in $A$ such that 
\[
\lim_{n\to\infty}\max_{\tau\in T(A)}\lvert \tau(c_n^m)-1/k\rvert=0
\]
for any $m\in\N$ and $\psi(e)=(c_n)_n$, 
where $e$ is a minimal projection in $M_k$. 
\end{lemma}

\begin{proof}
Let $\{\tau_1,\tau_2,\dots,\tau_N\}$ be the set of extremal points of $T(A)$. 
Set $\tau=N^{-1}\sum_{i=1}^N\tau_i$ and 
let $\pi$ be the GNS representation associated with $\tau\in T(A)$. 
Clearly $\tau_i$ and $\tau$ extend to tracial states on $\pi(A)''$. 
In what follows, we regard $A$ as a subalgebra of $\pi(A)''$ and omit $\pi$. 
Since $A$ is nuclear, 
$A''$ is isomorphic to the direct sum of 
$N$ copies of the AFD II$_1$-factor $\mathcal{R}$. 
In particular, $A''\bar\otimes\mathcal{R}\cong A''$. 
Hence we have a sequence of matrix units $E_{i,j,n}\in A''$ for $M_k$ 
such that 
\[
\lim_{n\to\infty}\lVert[E_{i,j,n},x]\rVert_\tau=0
\]
holds for any $x\in A''$. 
Define a unitary $U_n\in A''$ by 
\[
U_n=\sum_{i=1}^kE_{i,i+1,n}, 
\]
where $i+1$ is understood modulo $k$. 
By \cite[Lemma 2.1]{Sat2}, 
we can find a central sequence $(e_n)_n$ of positive contractions in $A$ and 
a central sequence $(u_n)_n$ of unitaries in $A$ such that 
\[
\lim_{n\to\infty}\lVert e_n-E_{1,1,n}\rVert_\tau=0\quad\text{and}\quad 
\lim_{n\to\infty}\lVert u_n-U_n\rVert_\tau=0. 
\]
Then we have 
\[
\lim_{n\to\infty}\lVert\Ad u_n^j(e_n)e_n\rVert_\tau=0
\]
for every $j=1,2,\dots,k-1$. 
From Lemma \ref{LemOrthogonality}, we may assume that $(e_n)_n$ and $(u_n)_n$ satisfy 
\[
\lim_{n\to\infty}\lVert\Ad u_n^j(e_n)e_n\rVert=0. 
\]
It follows from \cite[Proposition 2.4]{RW} that 
there exists a completely positive contractive order zero map 
$\psi:M_k\to A_\infty$ such that $\psi(e)=(e_n)_n$, 
where $e$ is a minimal projection in $M_k$. 
Because $\tau_i\leq N\tau$ for any $i=1,2,\dots,N$, 
one has 
\begin{align*}
\lvert\tau_i(e_n^m)-1/k\rvert
&=\lvert\tau_i(e_n^m-E_{1,1,n})\rvert\\
&\leq\lVert e_n^m-E_{1,1,n}\rVert_{\tau_i}^{1/2}\\
&\leq N^{1/2}\lVert e_n^m-E_{1,1,n}\rVert_\tau^{1/2}\to0,\quad n\to\infty
\end{align*}
for any $m\in\N$. 
The proof is completed. 
\end{proof}

\begin{lemma}\label{LemDecompositionfn}
Let $A$ be a unital separable simple infinite-dimensional 
nuclear C$^*$-algebra with $T(A)\neq\emptyset$. 
Suppose that the conclusion of Lemma \ref{LemOrderzeroCP} holds for $A$. 
Then for any central sequence $(f_n)_n$ of positive contractions in $A$ and 
any $k\in\N$, 
there exist central sequences $(f_{i,n})_n$, $i=1,2,\dots,k$ 
of positive contractions in $A$ 
such that $(f_nf_{i,n})_n=(f_{i,n})_n$, 
$(f_{i,n}f_{j,n})_n=0$ for $i\neq j$, and 
\[
\lim_{m\to\infty}\liminf_{n\to\infty}\min_{\tau\in T(A)}\tau(f_{i,n}^m)
=\lim_{m\to\infty}\liminf_{n\to\infty}\min_{\tau\in T(A)}\tau(f_n^m)/k. 
\]
\end{lemma}

\begin{proof}
Set $c=\lim_m\liminf_n\min_\tau\tau(f_n^m)$. 
Take a finite subset $F\subset A$, $\varepsilon>0$, and $N\in\N$ arbitrarily. 
It suffices to show that 
there exist sequences $(f_{i,n})_n$, $i=1,2,\dots,k$ 
of positive contractions in $A$ 
satisfying $\limsup_n\lVert[f_{i,n},a]\rVert<\varepsilon$ for each $a\in F$, 
$\limsup_n\lVert f_nf_{i,n}-f_{i,n}\rVert<\varepsilon$, 
$\limsup_n\lVert f_{i,n}f_{j,n}\rVert<\varepsilon$ for $i\neq j$, and 
\[
\left\lvert\liminf_{n\to\infty}\min_{\tau\in T(A)}\tau(f_{i,n}^m)-c/k\right\lvert<\varepsilon
\]
for any $m\leq N$.
Let $l\in\N$ be such that 
$\lvert (t-1)t^l\rvert<\varepsilon$ for $t\in[0,1]$ and 
\[
\left\lvert\liminf_{n\to\infty}\min_{\tau\in T(A)}\tau(f_n^{2lm})
-c\right\rvert<\varepsilon/2 
\]
for any $m\leq N$. 
Because we assumed that the conclusion of Lemma \ref{LemOrderzeroCP} holds for $A$, 
we obtain positive contractions $e_i\in A$, $i=1,2,\dots,k$ such that 
$\lVert[e_i,a]\rVert<\varepsilon$ for $a\in F$, 
$\lVert e_ie_j\rVert<\varepsilon$ for $i\neq j$, and 
$\max_\tau\lvert\tau(e_i^m)-1/k\rvert<\varepsilon/4$ for any $m\leq N$. 

Set $f_{i,n}=f_n^le_if_n^l$, $i=1,2,\dots,k$. 
Clearly it follows that 
$\limsup_n\lVert[f_{i,n},a]\rVert<\varepsilon$ for $a\in F$ and 
$\lVert f_nf_{i,n}-f_{i,n}\rVert\leq\lVert f_nf_n^l-f_n^l\rVert<\varepsilon$ 
for $n\in\N$. 
For $i\neq j$ we have 
$\limsup_n\lVert f_{i,n}f_{j,n}\rVert
\leq\limsup_n\lVert e_if_n^{2l}e_j\rVert<\varepsilon$. 
By \cite[Lemma 4.6]{MS}, we have 
\[
\limsup_{n\to\infty}\max_{\tau\in T(A)}\lvert\tau(f_n^{lm}e_i^mf_n^{lm})-\tau(f_n^{2lm})/k\rvert
\leq2\max_\tau\lvert\tau(e_i^m)-k^{-1}\rvert<\varepsilon/2 
\]
for any $m\leq N$. Since $\lVert f_{i,n}^m-f_n^{lm}e_i^mf_n^{lm}\rVert\to0$ as $n\to\infty$, 
we conclude that 
\begin{align*}
\left\lvert\liminf_{n\to\infty}\min_{\tau\in T(A)}\tau(f_{i,n}^m)-c/k\right\lvert
&=\lim_{N\to\infty}\left\lvert\inf_{n>N}\min_\tau\tau(f_{i,n}^m)-c/k\right\lvert\\
&\leq\limsup_N\left\lvert\inf_{n>N}\min_\tau\tau(f_n^{lm}e_i^mf_n^{lm})
-\inf_{n>N}\min_\tau\tau(f_n^{2lm})/k\right\rvert\\
&\quad+\lim_N\left\lvert\inf_{n>N}\min_\tau\tau(f_n^{2lm})/k
-c/k\right\lvert\\
&<\varepsilon/2+\varepsilon/2k\leq\varepsilon 
\end{align*}
for any $m\leq N$.
\end{proof}

We are now ready to prove (ii)$\Rightarrow$(iii) of Theorem \ref{ThmMain}. 
\begin{proof}[Proof of (ii)$\Rightarrow$(iii) of Theorem \ref{ThmMain}.]
Let $\varphi$ be a completely positive map from $A$ to $A$. 
We would like to show that 
$\varphi$ can be excised in small central sequences. 
Let $(e_n)_n$, $(f_n)_n$ be as in Definition \ref{DefExcision}. 
By Lemma \ref{LemCPmap} 
we may assume that there exist a pure state $\omega$ on $A$ and 
$c_{l,i},d_i\in A$, $l,i=1,2,\dots,N$, such that 
\[
\varphi(a)=\sum_{l=1}^N\sum_{i,j=1}^N
\omega(d_i^*ad_j)c_{l,i}^*c_{l,j},\quad a\in A. 
\]
Set 
$\displaystyle\varphi_l(a)=\sum_{i,j=1}^N\omega(d_i^*ad_j)c_{l,i}^*c_{l,j}$, 
$a\in A$ so that $\varphi=\varphi_1+\varphi_2+\dots+\varphi_N$. 

Applying Lemma \ref{LemDecompositionfn} to $(f_n)_n$, 
we have central sequences $(f_{l,n})_n$, $l=1,2,...,N$, 
of positive contractions in $A$ satisfying 
$(f_{l,n}f_n)_n=(f_{l,n})_n$, $(f_{l,n}f_{l',n})_n=0$ for $l\neq l'$, and 
\[
\lim_{m\to\infty}\liminf_{n\to\infty}
\min_{\tau\in T(A)}\tau(f_{l,n}^m)>0. 
\]
Applying Proposition \ref{PropMain} to 
$\varphi_l$, $(e_n)_n$, and $(f_{l,n})_n$, 
we obtain a sequence $(s_{l,n})_n$ in $A$ such that 
\[
\lim_{n\to\infty}\lVert s_{l,n}^*as_{l,n}-\varphi_l(a)e_n\rVert=0,\ a\in A,\quad 
\lim_{n\to\infty}\lVert f_{l,n}s_{l,n}-s_{l,n}\rVert=0. 
\]
We define $\displaystyle s_n=\sum_{l=1}^Ns_{l,n}$ for $n\in\N$. 
Since $\limsup_n\lVert s_{l,n}\rVert\leq\lVert\varphi_l(1)\rVert$, 
it follows that 
\begin{align*}
\lVert f_ns_n-s_n\rVert
&\leq\sum_{l=1}^N\lVert f_ns_{l,n}-s_{l,n}\rVert\\
&\leq\sum_{l=1}^N\lVert f_n\rVert\cdot\lVert s_{l,n}-f_{l,n}s_{l,n}\rVert
+\lVert f_nf_{l,n}-f_{l,n}\rVert\cdot\lVert s_{l,n}\rVert\to0,\quad n\to\infty. 
\end{align*}
If $\l\neq l'$, then 
\[
\lim_{n\to\infty}\lVert s_{l,n}^*as_{l',n}\rVert
=\lim_{n}\lVert s_{l,n}^*f_{l,n}af_{l',n}s_{l',n}\rVert=0
\]
for any $a\in A$. 
Therefore we conclude that 
\[
\lim_{n\rightarrow\infty}\lVert s_n^*as_n-\varphi(a)e_n\rVert
=\lim_n
\left\lVert\sum_{l=1}^N s_{l,n}^*as_{l,n}-\varphi_l(a)e_n\right\rVert=0. 
\]
\end{proof}

\section{\texorpdfstring{Proof of (iii)$\Rightarrow$(iv)$\Rightarrow$(i) of Theorem \ref{ThmMain}}{Proof of (iii) $\Rightarrow$ (iv) $\Rightarrow$ (i) of Theorem 1.1}}\label{Sec4}
In this section 
we prove (iii)$\Rightarrow$(iv)$\Rightarrow$(i) of Theorem 1.1. 
First, 
let us recall the definition of property (SI) from \cite{MS}. 

\begin{definition}[{\cite[Definition 4.1]{MS}}]\label{DefSI}
Let $A$ be a separable C$^*$-algebra with $T(A)\neq\emptyset$. 
We say that $A$ has {\it property (SI)} 
when for any central sequences $(e_n)_n$ and $(f_n)_n$ 
of positive contractions in $A$ satisfying 
\[
\lim_{n\to\infty}\max_{\tau\in T(A)}\tau(e_n)=0,\quad 
\lim_{m\to\infty}\liminf_{n\to\infty}\min_{\tau\in T(A)}\tau(f_n^m)>0, 
\]
there exists a central sequence $(s_n)_n$ in $A$ such that 
\[
\lim_{n\to\infty}\lVert s_n^*s_n-e_n\rVert=0,\quad 
\lim_{n\to\infty}\lVert f_ns_n-s_n\rVert=0. 
\]
\end{definition}

\begin{proof}[Proof of (iii)$\Rightarrow$(iv) of Theorem \ref{ThmMain}.]
Let $(e_n)_n$ and $(f_n)_n$ be as in Definition \ref{DefSI}. 
By the assumption of (iii), 
$\id_A$ can be excised in small central sequences. 
Thus we have $s_n\in A$, $n\in\N$ such that 
$\lVert s_n^*as_n-ae_n\rVert\to0$ for any $a\in A$ and 
$\lVert f_ns_n-s_n\rVert\to0$. 
Since $A$ is unital, we get $\lVert s_n^*s_n-e_n\rVert\to0$. 
Also, for any $a\in A$ we obtain 
\[
\limsup_{n\to\infty}\lVert[s_n,a]\rVert^2
=\limsup_{n\to\infty}
\lVert a^*s_n^*s_na-a^*s_n^*as_n-s_n^*a^*s_na+s_n^*a^*as_n\rVert=0, 
\]
which means that $(s_n)_n$ is central. 
\end{proof}

\begin{proof}[Proof of (iv)$\Rightarrow$(i) of Theorem \ref{ThmMain}.]
By Lemma \ref{LemOrderzeroCP}, 
we obtain central sequences $(c_{i,n})_n$ in $A$, $i=1,2,\dots,k$, 
such that $(c_{i,n}c_{j,n}^*)_n=\delta_{i,j}(c_{1,n}^2)_n$,
\[
\lim_{n\to\infty}\max_{\tau\in T(A)}
\lvert\tau(c_{1,n}^m)-1/k\lvert=0,\quad m\in\N, 
\]
and $c_{1,n}$ is a positive contraction for all $n\in\N$. 
Let $(e_n)_n$ be a central sequence of positive contractions in $A$
such that $\displaystyle (e_n)_n=(1-\sum_{i=1}^kc_{i,n}^*c_{i,n})_n$. 
Then we have 
\[
\limsup_{n\to\infty}\max_{\tau\in T(A)}\tau(e_n)
=\limsup_n\max_\tau\tau(1-\sum_{i=1}^kc_{i,n}^*c_{i,n})
=\limsup_n\max_\tau1-k\tau(c_{1,n}^2)=0
\]
and 
\[
\lim_{m\to\infty}\liminf_{n\to\infty}\min_{\tau\in T(A)}\tau(c_{1,n}^m)=1/k>0. 
\]
Thanks to property (SI), 
we obtain a central sequence $(s_n)_n$ in $A$ such that 
$(s_n^*s_n+\sum c_{i,n}^*c_{i,n})_n=1$ and $(c_{1,n}s_n)_n=(s_n)_n$, 
which means that $\{(c_{i,n})_n\}_{i=1}^k\cup\{(s_n)_n\}\subset A_\infty$ 
satisfies relation $\mathcal{R}_k$ defined in \cite[Section 2]{Sat}. 
It follows from \cite[Proposition 2.1]{Sat} that 
there exists a unital homomorphism 
from the prime dimension drop algebra $I(k,k+1)$ to $A_\infty$. 
The Jiang-Su algebra $\mathcal{Z}$ is 
an inductive limit of such $I(k,k+1)$'s. 
By \cite[Proposition 2.2]{TW}, 
we can conclude that $A\otimes\mathcal{Z}\cong A$. 
\end{proof}

In the same way as the proof above, we can show the following. 
Notice that 
we do not need the assumption of finitely many extremal traces 
for this theorem. 

\begin{theorem}\label{ThmSI}
Let $A$ be a unital separable simple infinite-dimensional 
 nuclear C$^*$-algebra with $T(A)\neq\emptyset$. 
Suppose that $A$ is $\mathcal{Z}$-absorbing. 
Then any completely positive map from $A$ to $A$ 
can be excised in small central sequences. 
Moreover, $A$ has property (SI). 
\end{theorem}
\begin{proof}
By \cite[Corollary 4.6]{Ror}, $A$ has strict comparison. 
Since $\mathcal{Z}$ is a unital separable simple infinite-dimensional 
 nuclear C$^*$-algebra with a unique trace, 
Lemma \ref{LemOrderzeroCP} is valid for $\mathcal{Z}$. 
Hence the conclusion of Lemma \ref{LemOrderzeroCP} also holds 
for $A\cong A\otimes\mathcal{Z}$. 
Then the proof of (ii)$\Rightarrow$(iii) of Theorem \ref{ThmMain} (see Section 3) 
works for $A$, and whence any completely positive map from $A$ to $A$ 
can be excised in small central sequences. 
By the proof of (iii)$\Rightarrow$(iv) of Theorem \ref{ThmMain}, 
we can conclude that $A$ has property (SI). 
\end{proof}

\section{C$^*$-algebras with tracial rank zero}\label{Sec5}
In this section we prove that 
any unital separable simple nuclear infinite-dimensional C$^*$-algebra 
with tracial rank zero 
is approximately divisible (Theorem \ref{TR0>appdv}). 

\begin{lemma}\label{TR0>centralMk}
Let $A$ be a unital separable simple infinite-dimensional C$^*$-algebra 
with tracial rank zero and let $k\in\N$. 
There exists a sequence $(\varphi_n)_n$ of homomorphisms 
from $M_k$ to $A$ such that 
$(\varphi_n(x))_n$ is a central sequence for any $x\in M_k$ and 
\[
\lim_{n\to\infty}\max_{\tau\in T(A)}\tau(1-\varphi_n(1))=0. 
\]
\end{lemma}
\begin{proof}
Let $C$ be 
a unital simple infinite-dimensional C$^*$-algebra with real rank zero. 
We first claim that for any $\varepsilon>0$ 
there exists a homomorphism $\varphi:M_k\to C$ such that 
$\tau(1-\varphi(1))<\varepsilon$ for every $\tau\in T(C)$. 
Choose $m\in\N$ so that $k/2^m$ is less than $\varepsilon$. 
By \cite[Theorem 1.1 (i)]{Zha}, 
there exists a partition of unity $1=p_1+p_2+\dots+p_{2^m}+q$ 
consisting of projections in $C$ such that 
$p_1$ is Murray-von Neumann equivalent to $p_i$ for all $i=1,2,\dots,2^m$ 
and $q$ is Murray-von Neumann equivalent to a subprojection of $p_1$. 
There exists a unital homomorphism from $M_{2^m}$ to $(1-q)C(1-q)$ and 
$\tau(q)<2^{-m}$ for any $\tau\in T(C)$. 
It follows that there exists a homomorphism $\varphi:M_k\to C$ such that 
$\tau(1-\varphi(1))\leq2^{-m}(k{-}1)+\tau(q)<2^{-m}k<\varepsilon$. 

We now prove the statement. 
Since $A$ has tracial rank zero, 
there exist a sequence of projections $e_n\in A$, 
a sequence of finite dimensional subalgebras $B_n$ of $A$ with $1_{B_n}=e_n$ 
and a sequence of unital completely positive maps $\pi_n:A\to B_n$ 
such that the following hold. 
\begin{itemize}
\item $\lVert[a,e_n]\rVert\to0$ as $n\to\infty$ for any $a\in A$. 
\item $\lVert\pi_n(a)-e_nae_n\rVert\to0$ as $n\to\infty$ for any $a\in A$. 
\item $\tau(1-e_n)<1/2n$ for all $\tau\in T(A)$. 
\end{itemize}
Choose a family of mutually orthogonal minimal projections 
$p_{n,1},p_{n,2},\dots,p_{n,k_n}$ of $B_n$ 
so that $e_nAe_n\cap B_n'\cong\bigoplus_ip_{n,i}Ap_{n,i}$. 
As $A$ has real rank zero, so does $p_{n,i}Ap_{n,i}$. 
It follows from the claim above that 
we can find a homomorphism $\varphi_{n,i}:M_k\to p_{n,i}Ap_{n,i}$ such that 
$\tau(p_{n,i}-\varphi_{n,i}(1))$ is arbitrarily small for all $\tau\in T(A)$. 
By taking a direct sum of $\varphi_{n,i}$'s, 
we get a homomorphism $\varphi_n:M_k\to e_nAe_n\cap B_n'$ 
such that $\tau(e_n-\varphi_n(1))<1/2n$ for all $\tau\in T(A)$. 
It is easy to see that 
$(\varphi_n(x))_n$ is a central sequence for any $x\in M_k$ 
and $\tau(1-\varphi_n(1))<1/n$ for every $\tau\in T(A)$. 
The proof is completed. 
\end{proof}

\begin{lemma}
Let $A$ be a unital separable simple nuclear infinite-dimensional 
C$^*$-algebra with tracial rank zero. 
Then any completely positive map from $A$ to $A$ can be excised in small central sequences. 
\end{lemma}
\begin{proof}
By \cite[Theorem 3.7.2]{Lin} and \cite[Corollary 3.10]{Perera}, $A$ has strict comparison. 
Then we can prove this lemma 
in the same way as the proof of (ii)$\Rightarrow$(iii) of Theorem 1.1 
(see Section 3), 
by using the lemma above instead of Lemma \ref{LemOrderzeroCP}. 
\end{proof}

\begin{lemma}
Let $A$ be a unital separable simple nuclear infinite-dimensional 
C$^*$-algebra with tracial rank zero. 
Then $A$ has property (SI). 
\end{lemma}
\begin{proof}
This follows from the lemma above and 
the proof of (iii)$\Rightarrow$(iv) of Theorem 1.1 (see Section 4). 
\end{proof}

\begin{theorem}\label{TR0>appdv}
Let $A$ be a unital separable simple nuclear infinite-dimensional 
C$^*$-algebra with tracial rank zero. 
Then $A$ is approximately divisible. 
In particular, $A$ is $\mathcal{Z}$-absorbing. 
\end{theorem}
\begin{proof}
In order to prove that $A$ is approximately divisible, 
it suffices to construct 
a unital homomorphism from $M_2\oplus M_3$ to $A_\infty$ 
(\cite[Proposition 2.7]{BKR}). 
By Lemma \ref{TR0>centralMk}, 
there exists a sequence $(\varphi_n)_n$ of homomorphisms 
from $M_2$ to $A$ such that 
$(\varphi_n(x))_n$ is a central sequence for any $x\in M_2$ and 
\[
\lim_{n\to\infty}\max_{\tau\in T(A)}\tau(1-\varphi_n(1))=0. 
\]
By the lemma above, $A$ has property (SI). 
It follows that 
there exists a central sequence $(s_n)_n$ such that 
\[
\lim_{n\to\infty}\lVert s_n^*s_n-(1-\varphi_n(1))\rVert=0,\quad 
\lim_{n\to\infty}\lVert\varphi_n(e_{11})s_n-s_n\rVert=0, 
\]
where $e_{11}\in M_2$ is a rank one projection in $M_2$. 
Hence there exists a unital homomorphism from $M_2\oplus M_3$ to $A_\infty$. 
Thus, $A$ is approximately divisible. 
By \cite[Theorem 2.3]{TW}, 
a unital separable approximately divisible C$^*$-algebra is 
$\mathcal{Z}$-absorbing. 
The proof is completed. 
\end{proof}

\end{document}